\documentclass[12pt]{amsart}
\usepackage{amsmath,amssymb,amsbsy,amsfonts,amsthm,latexsym, 
            amsopn,amstext,amsxtra,euscript,amscd,stmaryrd,mathrsfs,
            cite,array,mathtools,enumerate}

\usepackage{url}
\usepackage[colorlinks,linkcolor=blue,anchorcolor=blue,citecolor=blue,backref=page]{hyperref}

\usepackage[norefs,nocites]{refcheck}
\usepackage{color}
\usepackage{float}

\hypersetup{breaklinks=true}

\usepackage[english]{babel}

\usepackage{mathtools}
\usepackage{todonotes}

\usepackage{enumitem}

\usepackage{mathtools}
\usepackage{todonotes}
\usepackage{url}
\usepackage[colorlinks,linkcolor=blue,anchorcolor=blue,citecolor=blue,backref=page]{hyperref}
\usepackage{comment}

\begin{document}

\newtheorem{innercustomgeneric}{\customgenericname}
\providecommand{\customgenericname}{}
\newcommand{\newcustomtheorem}[2]{%
  \newenvironment{#1}[1]
  {%
   \renewcommand\customgenericname{#2}%
   \renewcommand\theinnercustomgeneric{##1}%
   \innercustomgeneric
  }
  {\endinnercustomgeneric}
}

\newtheorem{theorem}{Theorem}
\newtheorem{lemma}[theorem]{Lemma}
\newtheorem{claim}[theorem]{Claim}
\newtheorem{cor}[theorem]{Corollary}
\newtheorem{prop}[theorem]{Proposition}
\newtheorem{definition}{Definition}
\newtheorem{question}[theorem]{Open Question}
\newtheorem{example}[theorem]{Example}
\newtheorem{remark}[theorem]{Remark}
\newcustomtheorem{customthm}{Theorem}

\numberwithin{equation}{section}
\numberwithin{theorem}{section}

 \newcommand{\F}{\mathbb{F}}
\newcommand{\K}{\mathbb{K}}
\newcommand{\D}[1]{D\(#1\)}
\def\scr{\scriptstyle}
\def\\{\cr}
\def\({\left(}
\def\){\right)}
\def\[{\left[}
\def\]{\right]}
\def\<{\langle}
\def\>{\rangle}
\def\fl#1{\left\lfloor#1\right\rfloor}
\def\rf#1{\left\lceil#1\right\rceil}
\def\le{\leqslant}
\def\ge{\geqslant}
\def\eps{\varepsilon}
\def\mand{\qquad\mbox{and}\qquad}
\newcommand{\asum}{\sideset{}{^{\ast}}\sum}

\def\vec#1{\mathbf{#1}}

\newcommand{\lcm}{\operatorname{lcm}}

\def\bl#1{\begin{color}{blue}#1\end{color}} 

\newcommand{\QQ}{\mathbb{Q}}

\newcommand{\C}{\mathbb{C}}
\newcommand{\Fq}{\mathbb{F}_q}
\newcommand{\Fp}{\mathbb{F}_p}
\newcommand{\Disc}[1]{\mathrm{Disc}\(#1\)}
\newcommand{\Res}[1]{\mathrm{Res}\(#1\)}
\newcommand{\ord}{\mathrm{ord}}

\newcommand{\Q}{\mathbb{Q}}
\newcommand{\Z}{\mathbb{Z}}
\renewcommand{\L}{\mathbb{L}}

\newcommand{\Norm}{\mathrm{Norm}}

\def\cA{{\mathcal A}}
\def\cB{{\mathcal B}}
\def\cC{{\mathcal C}}
\def\cD{{\mathcal D}}
\def\cE{{\mathcal E}}
\def\cF{{\mathcal F}}
\def\cG{{\mathcal G}}
\def\cH{{\mathcal H}}
\def\cI{{\mathcal I}}
\def\cJ{{\mathcal J}}
\def\cK{{\mathcal K}}
\def\cL{{\mathcal L}}
\def\cM{{\mathcal M}}
\def\cN{{\mathcal N}}
\def\cO{{\mathcal O}}
\def\cP{{\mathcal P}}
\def\cQ{{\mathcal Q}}
\def\cR{{\mathcal R}}
\def\cS{{\mathcal S}}
\def\cT{{\mathcal T}}
\def\cU{{\mathcal U}}
\def\cV{{\mathcal V}}
\def\cW{{\mathcal W}}
\def\cX{{\mathcal X}}
\def\cY{{\mathcal Y}}
\def\cZ{{\mathcal Z}}

\def\fra{{\mathfrak a}} 
\def\frb{{\mathfrak b}}
\def\frc{{\mathfrak c}}
\def\frd{{\mathfrak d}}
\def\fre{{\mathfrak e}}
\def\frf{{\mathfrak f}}
\def\frg{{\mathfrak g}}
\def\frh{{\mathfrak h}}
\def\fri{{\mathfrak i}}
\def\frj{{\mathfrak j}}
\def\frk{{\mathfrak k}}
\def\frl{{\mathfrak l}}
\def\frm{{\mathfrak m}}
\def\frn{{\mathfrak n}}
\def\fro{{\mathfrak o}}
\def\frp{{\mathfrak p}}
\def\frq{{\mathfrak q}}
\def\frr{{\mathfrak r}}
\def\frs{{\mathfrak s}}
\def\frt{{\mathfrak t}}
\def\fru{{\mathfrak u}}
\def\frv{{\mathfrak v}}
\def\frw{{\mathfrak w}}
\def\frx{{\mathfrak x}}
\def\fry{{\mathfrak y}}
\def\frz{{\mathfrak z}}

\def\ov\QQ{\overline{\QQ}}
\def \brho{\boldsymbol{\rho}}

\def \fP {\mathfrak P}

\def \Prob{{\mathrm {}}}
\def\e{\mathbf{e}}
\def\ep{{\mathbf{\,e}}_p}
\def\epp{{\mathbf{\,e}}_{p^2}}
\def\em{{\mathbf{\,e}}_m}

\newcommand{\sR}{\ensuremath{\mathscr{R}}}
\newcommand{\sDI}{\ensuremath{\mathscr{DI}}}
\newcommand{\DI}{\ensuremath{\mathrm{DI}}}
\let\oldpmod\pmod
\renewcommand{\pmod}[1]{\hspace{-0.12cm}\oldpmod {#1}}

\newcommand{\Orb}[1]{\mathrm{Orb}\(#1\)}
\newcommand{\aOrb}[1]{\overline{\mathrm{Orb}}\(#1\)}

\def \Nm{{\mathrm{Nm}}}
\def \Gal{{\mathrm{Gal}}}

\newenvironment{notation}[0]{%
  \begin{list}%
    {}%
    {\setlength{\itemindent}{0pt}
     \setlength{\labelwidth}{1\parindent}
     \setlength{\labelsep}{\parindent}
     \setlength{\leftmargin}{2\parindent}
     \setlength{\itemsep}{0pt}
     }%
   }%
  {\end{list}}

\definecolor{dgreen}{rgb}{0.,0.6,0.}
\def\tgreen#1{\begin{color}{dgreen}{\it{#1}}\end{color}}
\def\tblue#1{\begin{color}{blue}{\it{#1}}\end{color}}
\def\tred#1{\begin{color}{red}#1\end{color}}
\def\tmagenta#1{\begin{color}{magenta}{\it{#1}}\end{color}}
\def\tNavyBlue#1{\begin{color}{NavyBlue}{\it{#1}}\end{color}}
\def\tMaroon#1{\begin{color}{Maroon}{\it{#1}}\end{color}}

\title[On Artin's conjecture on average]{On Artin's conjecture on average and short character sums} 

 \author[O.~Klurman]{Oleksiy Klurman}
 \address{School of Mathematics, University of Bristol, BS8 1QU, United Kingdom}
 \email{oleksiy.klurman@bristol.ac.uk}
 
  \author[I.~E.~Shparlinski]{Igor E. Shparlinski}
 \address{School of Mathematics and Statistics, University of New South Wales.
 Sydney, NSW 2052, Australia}
 \email{igor.shparlinski@unsw.edu.au}

\author[J. Ter\"av\"ainen] {Joni Ter\"av\"ainen}
\address{Department of Mathematics and Statistics, University of Turku, 20014
Turku, Finland}
\email{joni.p.teravainen@gmail.com}

\begin{abstract} Let $N_a(x)$ denote the number of primes up to $x$ for which the integer $a$ is a primitive root. We show that $N_a(x)$ satisfies the asymptotic predicted by Artin's conjecture for almost all $1\le  a\le  \exp((\log \log x)^2)$. This improves on a result of Stephens (1969). A key ingredient in the proof is a new short character sum estimate over the integers, improving on  the range of  a result of Garaev (2006). 
\end{abstract}  


\pagenumbering{arabic}

\maketitle

\section{Introduction}

\subsection{Background and set-up} 
Artin's famous conjecture on primitive roots asserts that any integer  $a \ne -1$, which is not a perfect square, 
is a primitive root for a set of primes of positive relative density, with the density depending on the arithmetic structure 
of $a$. 
This has been established conditionally under the generalised Riemann hypothesis (GRH) for Dedekind zeta functions of certain number fields, 
in a celebrated work of Hooley~\cite{Hool}.  

More precisely, if $a$ is an integer that is not $-1$ or a perfect square, let $b\in \mathbb{N}$ be the squarefree part of $a$ and let 
 $h\ge 1$ be  the largest integer such that $a$ is a perfect $h$-th power. We also define
 $$
A(h) =  \prod_{\ell\mid h}\left(1-\frac{1}{\ell-1}\right)\prod_{\ell\nmid h}\left(1-\frac{1}{\ell(\ell-1)}\right), 
$$
where $\ell$ runs over primes. 
Denoting 
$$
N_a(x)=\#\{p\le  x\colon~a\textnormal{ is a primitive root modulo } p\},
$$
by a result of Hooley~\cite[Section~7]{Hool} under GRH we have  
\begin{itemize}
\item  if $b \not \equiv 1 \pmod 4$, then
$$
N_a(x) = A(h) \pi(x)  + O_a\( \frac{x \log \log x}{(\log x)^2}\),
$$
\item If $b  \equiv 1 \pmod 4$, then
\begin{align*}
N_a(x) = A(h)\(1-\mu(b) \prod_{\substack{\ell\mid h\\\ell \mid b}} \frac{1}{\ell-2} 
 \prod_{\substack{\ell\nmid h\\\ell \mid b}} \frac{1}{\ell^2 -\ell-1}\) &\pi(x) \\ + O_a&\( \frac{x \log \log x}{(\log x)^2}\),
\end{align*}
\end{itemize}
where, as usual, $\mu(b)$ denotes the M{\"o}bius function, $\pi(x)$ is the prime counting function, 
and $O_\rho$ and $\ll_\rho$ indicates that the implied constants
may depend on the parameter $\rho$, see Section~\ref{sec:not} for an exact definition. 
We also refer to the exhaustive survey of Moree~\cite{Mor}
for a wide variety of other results and references, see also~\cite{JP,JPS,PeSh} for more recent developments and 
further references. 

In another direction, we mention the celebrated result of Heath-Brown~\cite{HB}, improving on a beautiful work of Gupta and Murty~\cite{GuMu}, which shows that Artin's conjecture holds for all but possibly two primes. Heath-Brown also shows that
$$
\#\{|a|\le  y\colon~\limsup_{x\to \infty}N_a(x)<\infty\}\ll (\log y)^2.    
$$
However, the methods of~\cite{GuMu,HB}  do not yield an asymptotic for $N_a(x)$ for almost all $a$ (and in fact the lower bounds are off from the conjectured magnitude by a factor of $(\log x)^{-1}$, since all the primes $p$ detected are such that $(p-1)/K$ has at most two prime factors for some small positive integer $K$).

Since  an unconditional proof  of Artin's conjecture still seems to be out of reach, it is interesting to study $N_a(x)$ for almost all $a$
(with an ultimate goal of reducing the amount of averaging).  
We observe that for a ``typical'' integer $a$, we have $h=1$, while $b$ is quite large, 
making the  main terms in the above asymptotic formulas for $N_a(x)$
to be $A\pi(x)$  in both cases,  with 
\begin{equation}
\label{eq:  Art Const}
A = A(1)= \prod_{\ell }\left(1- \frac{1}{ \ell(\ell-1)}\right) =0.373955 \ldots, 
\end{equation}
which is called {\it Artin's constant\/}. 

In particular, Stephens~\cite{Steph}, improving on a previous results of Goldfeld~\cite{Goldfeld}, 
established in 1969 the following almost-all result. 

\begin{customthm}{A}
Let $D\geq 1$, and let $x,y\geq 3$ satisfy 
\begin{equation}
\label{eq: Steph Range}
y \ge \exp\( 6 \( \log x \log  \log x\)^{1/2}\). 
\end{equation}
Then we have
\begin{equation}
\label{eq: Av Art}
\left|N_a(x)-A\pi(x)\right|\le  \frac{\pi(x)}{(\log x)^D}
\end{equation}
for all but $O_D\(y/(\log x)^{D}\)$ integers $a\in [-y,y]$. 
\end{customthm}

In this paper, we augment the ideas of Stephens~\cite{Steph} with arguments involving short character sums and the anatomy of integers and reduce quite significantly the range~\eqref{eq: Steph Range} of $y$ 
for which~\eqref{eq: Av Art} holds.

\begin{theorem}
\label{thm:Artin Aver}  Let $D\geq 10$, and let $x,y\geq 100$ satisfy 
\begin{equation}\label{eq:ylower}
y \ge \exp\( \frac{60(D+1)}{\log 2}\cdot \frac{\(\log  \log x\)^{2}}{\log \log \log x}\). 
\end{equation}
Then we have
$$
\left|N_a(x)-A\pi(x)\right|\le  \frac{\pi(x) }{(\log x)^D}
$$
for all but $O_D\(y/(\log x)^{D}\)$ integers $a\in [-y,y]$. 
\end{theorem}

\subsection{Short character sums}

As in~\cite{Steph}, our approach is based on bounds for short character sums
for almost all integer moduli. However, we can obtain cancellation in significantly shorter character sums, leading to an improved range in the application to Artin's conjecture. 
We refer to~\cite[Chapter~3]{IwKow} for the relevant background on character sums.

\begin{theorem}
\label{thm:short char sum} Let $x,y\geq 3$ and 
\begin{equation}
\label{eq: range C}
3 \le  \lambda\le  (\log y)/(\log \log y)^2, 
\end{equation}
and suppose that 
\begin{equation}
\label{eq: range y}
\exp\(20\lambda (\log(\lambda\log \log x)) \(\log \log x\)\)\le  y\le  x.
\end{equation}
Then, for all but at most $O(x^{0.49})$ natural numbers $q \le  x$, we have 
\begin{equation}\label{eq:charsum}
 \max_{\chi \in \cX_q^*} 
\left|\sum_{1 \le a \le y} \chi(a)\right|\ll y / (\log y)^{\lambda\log 2-1}, 
\end{equation}
where  $\cX_q^*$ denotes the set of all primitive
Dirichlet characters modulo $q$.  
\end{theorem}  

\begin{remark}
It is important to note that the implied constants in Theorem~\ref{thm:short char sum} 
do not depend on $\lambda,$ which we choose 
to grow with $x$ in the proof of  Theorem~\ref{thm:Artin Aver}; see~\eqref{eq:Cdef}. 
\end{remark} 

\begin{remark}
From the proof of Theorem~\ref{thm:short char sum} in 
Section~\ref{sec:proof char}, one can see that one could enlarge the range of $y$ in~\eqref{eq: range y} by replacing  $20\lambda$ with $7.721\lambda$ there,  
at the cost of increasing the size of the exceptional 
 set of natural numbers  from  $O(x^{0.49})$ to  $O(x^{1-\delta})$ for some small $\delta>0$.  On the other hand,  one can decrease the size of the
 exceptional set of natural numbers to $O(x^{c/K})$ (with an absolute constant $c > 0$) by replacing $20\lambda$ with $K\lambda$ in~\eqref{eq: range y} for $K$ a large constant. For our application, it is  helpful to have an exponent smaller than $1/2$. 
\end{remark} 

We remark that a result of Garaev~\cite[Theorem~9]{Gar} gives a  power-saving bound for the character sums in~\eqref{eq:charsum} 
for almost all moduli $q \le  x$, provided that 
$$y = \exp\(c \sqrt{\log x}\)$$
for a suitable constant $c > 0$. In fact,~\cite[Theorem~10]{Gar} gives more flexibility 
for larger values of $y$. Here we are mostly interested in small values 
of $y$, which are not covered by the results of~\cite{Gar}.

\subsection{Notation}
\label{sec:not} 

We recall that  the notations $U = O(V)$, $U \ll V$ and $ V\gg U$  
are equivalent to $|U|\leqslant c V$ for some positive constant $c$, 
which we take to be absolute unless indicated with subindices. 
For example, $O_D$ and $\ll_D$ both mean the the implied constant 
may depend on the parameter $D$.
We use $U\asymp V$ as a shorthand for $U\ll V\ll U$.  

The letter $p$, with or without subscripts, always denotes a prime number.

We use $\sum^\ast_{r\pmod q}$ to denote summation over the primitive residue classes modulo $q$. 

We denote by $\mathcal{X}_q$ the set of Dirichlet characters $\pmod q$ and by $\mathcal{X}_q^{*}$ the set of primitive Dirichlet characters $\pmod q$.

We use the standard notation $\mu(n)$, $\varphi(n)$ and $\Omega(n)$ for the M{\"o}bius function, the Euler function, and the number of prime divisors function (counted with multiplicities),  respectively. Furthermore, $\pi(x)$ denotes the number of primes $p \le x$.

Finally, we use $\# \cS$ to denote the cardinality of a finite set $\cS$.

\subsection{Acknowledgments}
The authors thank Kaisa Matom\"aki for a helpful discussion on character sums and the referee
for the careful reading of the manuscript. 

  During the preparation of this work  
I.S.\ was  supported by the Australian Research Council Grants  DP230100530 and DP230100534 and by a 
Knut and Alice Wallenberg Fellowship.  
J.T. was supported by European Union's Horizon
Europe research and innovation programme under Marie Sk\l{}odowska-Curie grant agreement No. 101058904, and Academy of Finland grant No.~362303.  

This work started while the authors were visiting   Institut Mittag-Leffler, Sweden, 
during the programme `Analytic Number Theory' in January--April of 2024, 
whose hospitality and support are gratefully acknowledged.

\section{Proof of the character sum bound} 
\label{sec:proof char}

\subsection{Preliminaries}
In this section, we prove Theorem~\ref{thm:short char sum}. We begin with a useful estimate for the number of integers with a given number of prime factors.

\begin{lemma}\label{le:nicolas} Let $x\geq 3$ and let $m$ be a positive integer. 
Then we have
\begin{itemize}
\item[(i)] for $3   \log \log x\le m \le   \(\log x\)/\log 2$,
$$
\frac{x}{2^m\log x}\ll \#\{n\le  x\colon~\Omega(n)=m\}\ll \frac{x}{2^m}\log \frac{x}{2^{m}} + 1; 
$$ 
\item[(ii)] for any $m \le   (\log x)/\log 2$,
$$
\#\{n\le  x\colon~\Omega(n)=m\}\gg \frac{x}{2^m\log x}.     
$$ 
\end{itemize}
\end{lemma}

\begin{proof}
If $(\log x)/\log 2-1/10<m\le (\log x)/\log 2$, then clearly $$\{n\le  x\colon~\Omega(n)=m\} =  \{2^m\}.$$ 
Hence the cardinality to be estimated is equal to $1$, so both  bounds (i) and~(ii) are trivial.

 For $3\log \log x\le  m\le (\log x)/\log 2-1/10$,  the  bound~(i), and thus~(ii) in this range follows from a result of 
 Nicolas~\cite{nicolas} (see also~\cite[Part~II, Equation~(6.30) and Exercise~217]{Ten}, which in fact gives an asymptotic formula
\begin{equation}\label{eq:Omega=m}
 \#\{n\le  x\colon~\Omega(n)=m\} = \(C  + O((\log y)^{-\eta})\) y \log y
\end{equation}
where $y = x/2^m$, $\eta>0$ is some absolute constant,  and the constant $C = 0.3786 \ldots$ is given by an explicit Euler product (note that $y \ge 2^{1/10} > 1$). Indeed, we note that~\eqref{eq:Omega=m} always gives the desired upper bound (even in the regime when $y \ll 1$. To extract a lower bound from~\eqref{eq:Omega=m}  we need to assume that $y \to \infty$. However since  $2^m \in  \{n\le  x\colon~\Omega(n)=m\}$ in the above range of $m$, for 
$y < \log \log x$ the lower bound is trivial, so we indeed get~(i).

Finally, to establish the lower bound~(ii) for $m < 3\log \log x$, we note that in this range 
\begin{align*}
\#\{n\le  x\colon~\Omega(n)=m\}  &\ge \#\{2^{m-1} p\colon~ p \le x/2^{m-1}\} \\
&= \pi(x/2^{m-1}) \gg \frac{x}{2^m\log (x/2^m)} \gg  \frac{x}{2^m\log x}, 
\end{align*}
which concludes the proof. 
 \end{proof}
 
 \begin{remark}
We remark that the coefficient $3$ in Lemma~\ref{le:nicolas}~(i) can be replaced with $2 + \delta$ with any fixed $\delta > 0$. 
\end{remark} 

\subsection{Proof of Theorem~\ref{thm:short char sum}}
We may assume that $x$ is larger than any given absolute constant.  We may also assume that $y\le  x^{0.34}$, say, since otherwise the result follows from the Burgess bound~\cite[Theorem~12.4]{IwKow}. By applying Garaev's result~\cite[Theorem~10]{Gar} with $\delta=1/10$, we may further assume that 
\begin{equation}\label{eq:ybound}
y\le  \exp((\log x)^{0.51}),     
\end{equation}
say.

For $w\in \mathbb{N}$, let $$\cA_w(y)=\{a\le  y\colon~\Omega(a)=w\}.$$ 
For $k\in \mathbb{N}$, we also denote 
$$
S_k(w, x, y) = \sum_{q \le  x} \max_{\chi \in \cX_q^*} 
\left|\sum_{a \in \cA_w(y)} \chi(a)\right|^{2k}.
$$

We first give a bound for $S_k(w,x,y)$. By replacing the maximum with summation over all primitive characters and applying the multiplicative large sieve~\cite[Theorem~7.13]{IwKow}, we have
\begin{equation}
\label{eq: S and T}
S_k(w,x,y)\le  \sum_{q \le  x} \sum_{\chi \in \cX_q^*} 
\left|\sum_{a \in \cA_w(y)} \chi(a)\right|^{2k}\ll(x^2+y^k)T_k(w,y),
\end{equation}
where 
$$
T_k(w, y) = \sum_{m\le  y^k}r_k(m,w,y)^2
$$ 
is the number of solutions to the equation 
$$
a_1\cdots a_k = b_1\cdots b_k, \qquad  a_1, b_1, \ldots, a_k, b_k  \in \cA_w(y). 
$$
By considering the number of choices for $(a_1,\ldots, a_k)$ for a given $k$-tuple $(b_1,\ldots, b_k)\in \cA_w(y)^k$, we have
$$
T_k(w, y) \le  \binom{kw} {\, \underbrace	{w, \ldots, w}_{k\ \text{times}}\, } \( \# \cA_w(y)\)^k,
$$
since each choice of $a_1, \ldots, a_k$
corresponds to partitioning the multi-set of the $kw$ (not necessary distinct) primes dividing $b_1\cdots b_k$ into 
$k$ groups of $w$ primes.
Using the  elementary inequalities 
$$
(n/e)^n \le n! \le  e^2 (n/e)^{n+1},
$$
we see that
$$
\binom{kw} {\, \underbrace	{w, \ldots, w}_{k\ \text{times}}\, } = \frac{(kw)!}{(w!)^k} \le e^2 \frac{(kw/e)^{kw + 1}}{ (w/e)^{kw}}
=  e  w k^{kw+1}.
$$ 

Therefore 
$$
T_k(w, y) \ll w k^{kw+1}  \( \# \cA_w(y)\)^k , 
$$
which after the substitution in~\eqref{eq: S and T} implies
\begin{equation}
\label{eq: Bound S}
S_k(w, x, y)  \ll  w k^{kw+1} \(x^2 + y^k\) \( \# \cA_w(y)\)^k. 
\end{equation}

We now define the integer $k\ge 1$ by the inequalities 
\begin{equation}\label{eq:kbound}
y^{k} < x^2 \le y^{k+1}.
\end{equation}
For a given  $\delta >0$, we consider the set
$$
\cE(w, x, y; \delta) = \left\{q\le  x\colon
\max_{\chi \in \cX_q^*} 
\left|\sum_{a \in \cA_w(y)} \chi(a)\right| \ge \delta\cdot \# \cA_w(y) \right\}.
$$
From~\eqref{eq: Bound S} and our choice of $k$, we derive 
\begin{equation}
\label{eq: E kwy delta}
\begin{split}
\#\cE(w, x, y; \delta) & \ll  \(\# \cA_w(y) \delta\)^{-2k}  y^{k+1} \( \# \cA_w(y)\)^k w k^{kw+1}\\
 & \ll  kwy \( \frac{k^w y}{ \delta^{2} \# \cA_w(y)}\)^k .
\end{split}
\end{equation}
We impose the restriction
\begin{equation}
\label{eq: small w}
w \le   \lambda \log \log y.
\end{equation}
Note that by the assumption~\eqref{eq: range C}  on $\lambda$ we have $w\le  (\log y)/(\log \log y)$. 
By Lemma~\ref{le:nicolas}~(ii), under the condition~\eqref{eq: small w} we have $$
\# \cA_w(y)  \gg \frac{y}{2^w} \frac{1}{\log y}.
$$ 
Hence, taking $\delta =  (\log y)^{-\lambda\log 2}$ in the bound~\eqref{eq: E kwy delta}, we obtain
\begin{equation}
\label{eq: E kwyB}
\# \cE(w, x, y; (\log y)^{-\lambda\log 2})  \ll  kwy \(  k^w  (\log y)^{4\lambda\log 2}\)^k , 
\end{equation}

Next, recalling the choice of $k$ from~\eqref{eq:kbound}, we  see that 
\begin{equation}
\label{eq: k asymp}
 2\frac{\log x}{\log y} -1 \le k  < 2\frac{\log x}{\log y}.
\end{equation}  
Hence, we derive from~\eqref{eq: E kwyB} that 
$$
\#  \cE(w, x, y;(\log y)^{-\lambda\log 2})  \ll  x^{\xi}, 
$$
where, using the upper bound on $k$ from~\eqref{eq: k asymp}, we conclude
$$ \xi  =   2\frac{w  \log \log x  + 4\lambda(\log 2)(\log \log y)}{\log y} + O\left(\frac{\log (kwy)}{\log x}\right).$$
Recalling our assumptions~\eqref{eq: range C}, \eqref{eq:ybound}, \eqref{eq: small w}  and the estimate~\eqref{eq: k asymp} again,  we see that 
$$
 \frac{\log (kwy)}{\log x}\ll  \(\log x\)^{-0.48}.
$$
Now, from~\eqref{eq: range C},~\eqref{eq: small w}, and then also from~\eqref{eq: range y}, we derive
\begin{align*}
\xi & \le   2     \frac{w  \log  \log x  +  4\lambda (\log 2)(\log \log y)}{\log y}+ O\(\(\log x\)^{-0.48}\) \\
& \le   2 \lambda  \frac{\(\log  \log x\)\(\log \log y\)}{\log y}+ O\(\(\log x\)^{-0.48}+\(\log \log y\)^{-1}\)  \\
& \le 2\lambda\frac{\log(20\lambda\log(\lambda\log \log x))+\log \log \log x}{20\lambda\log(\lambda \log \log x)}\\
& \qquad  \qquad  \qquad  \qquad  \qquad  \qquad +  O\(\(\log x\)^{-0.48}+ \(\log \log y\)^{-1}\)\\
&\le \frac{2\log(20\lambda)}{20\log \lambda}+\frac{2\log \log (\lambda \log \log x)}{20\log(\lambda \log \log x)}+\frac{2}{20}+\frac{1}{1000}\\
&\le \frac{2\log 60}{20\log 3}+\frac{1}{1000}+\frac{2}{20}+\frac{1}{1000}\\
&\le 0.485,
 \end{align*}
provided that $x$ (and hence $y$) is large enough in absolute terms.
This implies that
\begin{equation}
\label{eq: Ew-Bound}
\# \cE(w, x, y; (\log y)^{-\lambda\log 2})  \le x^{0.485} 
\end{equation}
for each $w$ satisfying~\eqref{eq: small w} and any $x$ that is large enough in absolute terms. 

It now remains to estimate the quantity
$$
F_{\lambda}(y) = \#\{a\le  y\colon~\Omega(a) > \lambda \log \log y\}.
$$

Using Lemma~\ref{le:nicolas}~(i), which applies since by~\eqref{eq: range C}, 
we have  $\lambda\ge 3$,  we derive
\begin{equation}\begin{split}
\label{eq: F-Bound}
F_{\lambda}(y)  &\le    \sum_{\lambda\log \log y< m\le  (\log y)/\log 2}\#\{a\le  y\colon~\Omega(a)=m\}\\
&\ll y\log y\sum_{m>\lambda\log \log y}2^{-m} +\log y\\
&\ll  y(\log y)^{1-\lambda\log 2} + \log y \ll   y(\log y)^{1-\lambda\log 2},
\end{split}
\end{equation}
since under~\eqref{eq: range C} we clearly have the bound $(\log y)^{\lambda\log 2} = y^{o(1)}$.

  Let
 $$
 \cE(x, y) = \bigcup_{0\le w \le \lambda \log \log y}  \cE(w, x, y;(\log y)^{-\lambda\log 2}) , 
 $$
which by~\eqref{eq: Ew-Bound} and~\eqref{eq: range C}  is of cardinality 
\begin{align*}
\#  \cE(x, y) & \le     x^{0.485} \lambda \log \log y \\
& \le
 x^{0.485}   \frac{\log y}{\log \log y} 
\ll  x^{0.49}.
\end{align*}
 Then, for any natural numbers $w \le \lambda \log \log y$ and 
 $q \le x$ with $q \not \in  \cE(x, y)$, we have
 \begin{equation}\label{eq:chimax}
 \max_{\chi\in \mathcal{X}_q^{*}}\left|\sum_{a \in \cA_w(y)} \chi(a)\right| \ll y /(\log y)^{\lambda\log 2}. 
 \end{equation}
 Writing 
$$
\left|\sum_{a \le y} \chi(a) \right|  \le \sum_{0\le w \le \lambda \log \log y}  \left|\sum_{a \in \cA_w(y)} \chi(a)\right| + F_{\lambda}(y)
$$
and using the bounds~\eqref{eq: F-Bound} and~\eqref{eq:chimax}, and the assumed upper bound in~\eqref{eq: range C}  for $\lambda$,
we  conclude the proof.

\section{Proof of Theorem~\ref{thm:Artin Aver}} 

\subsection{Moments of \texorpdfstring{$N_a(x)$}{Na(x)}} 
The main task in this section is to prove the following second moment estimate.

\begin{prop}\label{prop:L2} Let the assumptions be as in Theorem~\ref{thm:Artin Aver}. Then we have
$$
\sum_{|a|\le  y}\left|N_a(x)-A\pi(x)\right|^2\ll_D y\frac{\pi(x)^2 }{(\log x)^{3D}}.
$$
\end{prop}

The deduction of Theorem~\ref{thm:Artin Aver} from this is immediate via  Chebyshev's inequality.

For the proof of Proposition~\ref{prop:L2}, we establish the following mean value estimate.

\begin{prop}\label{prop:L1}
Let the assumptions be as in Theorem~\ref{thm:Artin Aver}. Then we have
$$
\sum_{|a|\le  y}N_a(x)=2Ay\pi(x)+O_D\left(y\frac{\pi(x)}{(\log x)^{3D}}\right).
$$    
\end{prop}

\begin{proof}
Let us fix some $D > 10$. We may assume that $x,y$ are larger than any given absolute constant.
We are going to use Theorem~\ref{thm:short char sum}  with 
\begin{equation}\label{eq:Cdef}
\lambda=\frac{3D+2}{\log 2} \cdot\frac{\log \log (x^2)}{\log \log y} + \frac{1}{\log 2}.
\end{equation}
Note that by~\eqref{eq:ylower} the condition~\eqref{eq: range C} holds. 
Also note that for this choice of $\lambda$ we have
$$
y>\exp( 20\lambda \log(\lambda (\log \log (x^2)))\log \log(x^2)).
$$
Now, by Theorem~\ref{thm:short char sum},
we see that there is a set $\mathcal{E}$ of size $\#\mathcal{E}\ll x^{1/2}$ (say) such that for all primes $p\in [1,x]\setminus \mathcal{E}$ we have
\begin{equation}
\label{eq: small Spy}
 \max_{\chi \in \cX_p^*} 
\left|\sum_{1 \le a \le y} \chi(a)\right|\le y / \(\log x\)^{3D+2}.
\end{equation}  
Let $$\mathcal{P}=\{p\le  x\}\setminus \mathcal{E}.$$

Let us denote by $\cG_p$ the set of primitive roots modulo $p$. We can write 
\begin{equation}
\label{eq: Good Primes}
\begin{split}
\sum_{-y \le a \le y} N_a(x) & = \sum_{-y \le a \le y} \, \sum_{\substack{p\le x\\ a\in \cG_p}} 1 = 
\sum_{p\le x} \, \sum_{\substack{-y \le a \le y\\a \in \cG_p}} 1.\\
&= 
\sum_{p\in \mathcal{P}} \, \sum_{\substack{-y \le a \le y\\a \in \cG_p}} 1+O(x^{1/2}y).
\end{split}
\end{equation}

Using the standard inclusion-exclusion argument to detect primitive roots (see, for example,~\cite[Problem~5.14]{LN} or~\cite[Proposition~2.2]{Nark}), we see that 
for any integer $a$ we have
 \begin{equation}\label{eq:primroot}
\frac{\varphi(p-1)}{p-1} \sum_{t \mid p-1} \frac{\mu(t) }{ \varphi(t)} 
 \sum_{\substack{\chi \in \cX_p \\ \ord(\chi) = t}} \chi(a) = 
  \begin{cases} 1 & \text{if}\ a\in \cG_p,\\
       0, & \text{otherwise,}\\
\end{cases}
\end{equation}
where $\ord(\chi)$ denotes the order of $\chi$ in the group of characters $\cX_p$.

Separating the contribution of the principal character, corresponding to $t =1$ in~\eqref{eq:primroot}, we arrive at 
\begin{equation}
\label{eq: Asymp 1}
 \sum_{p \in \mathcal{P}} \, \sum_{\substack{-y \le a \le y\\a \in \cG_p}} 1 =  \#([-y,y]\cap \mathbb{Z})\cdot \sum_{p \in \cP} \frac{\varphi(p-1)}{p-1} + O(E+y \log y)
\end{equation}
with 
$$
E=  \sum_{p \in \mathcal{P}}
\sum_{\substack{t \mid p-1\\t > 1}} \frac{|\mu(t)|}{ \varphi(t)} 
 \sum_{\substack{\chi \in \cX_p^* \\ \ord(\chi) = t}}   \left|\sum_{-y\le a \le y} \chi(a)\right|,
$$
where the $O(y \log y)$ term comes from  the contribution
$$
\sum_{p \le y} \(2y/p + O(1)\) = 2y  \sum_{p \le y} 1/p + O(1) \ll y \log y
$$
of the principal characters modulo $p$ on values $a \equiv 0 \pmod p$.

First, we note that 
$$
 \sum_{p \in \cP} \frac{\varphi(p-1)}{p-1}    =    \sum_{p \le x} \frac{\varphi(p-1)}{p} 
 +  O(x^{1/2}).
$$
Since by~\cite[Lemma~1]{Steph} we have 
\begin{equation}\label{eq:phi}
 \sum_{p \le x} \frac{\varphi(p-1)}{p} = A \pi(x) +  O_D\left(\frac{\pi(x)}{(\log x )^{3D}}\right), 
\end{equation}
with $A$ given by~\eqref{eq:  Art Const}, we can rewrite~\eqref{eq: Asymp 1} as
\begin{equation}
\label{eq: Asymp 2}
 \sum_{p \in \cP} \, \sum_{\substack{-y \le a \le y\\a \in \cG_p}} 1 =   2A y\pi(x) +  O_D\left(\frac{y\pi(x)}{(\log x )^{3D}} +E\right).
\end{equation}
Thus it remains to estimate $E$. 

Using~\eqref{eq: small Spy}, we bound 
\begin{equation}\label{eq:Ebound}
E\ll_D   \frac{y}{\(\log x\)^{3D+2}}  \sum_{p \in \cP}\sum_{\substack{t \mid p-1\\t > 1}} \frac{\left| \mu(t) \right|}{ \varphi(t)}
 \sum_{\substack{\chi \in \cX_p^* \\ \ord(\chi) = t}}   1.
\end{equation}
Since the group $\cX_p$ is cyclic of order $p-1$, for each divisor $t \mid p-1$ 
there are $\varphi(t)$ characters $\chi \in \cX_p^*$ with $\ord(\chi) = t$.
Therefore we have
$$
\sum_{\substack{t \mid p-1\\t > 1}} \frac{\left| \mu(t) \right|}{ \varphi(t)}  \sum_{\substack{\chi \in \cX_p^* \\ \ord(\chi) = t}}   1
= \sum_{\substack{t \mid p-1\\t > 1}}|\mu(t)|=2^{\omega(p-1)}-1. 
$$
Recalling~\eqref{eq:Ebound}, we see that 
$$
E \ll_D   \frac{y}{\(\log x\)^{3D+1}} \sum_{p\le  x}2^{\omega(p-1)}\ll    \frac{yx}{\(\log x\)^{3D+2}}\ll \frac{y\pi(x)}{(\log x)^{3D}} 
$$
by the Titchmarsh divisor estimate 
\begin{equation}\label{eq:titchmarsh}
\sum_{p\le  x}2^{\omega(p-1)}\ll x
\end{equation}
(see~\cite{Titchmarsh}) and the prime number theorem. This together with~\eqref{eq: Good Primes} and~\eqref{eq: Asymp 2} concludes the proof. 
\end{proof}

\subsection{Proof of Proposition~\ref{prop:L2}}
Let $D\ge 10$ be fixed. Expanding out the square and applying Proposition~\ref{prop:L1}, we have
\begin{align*}
\sum_{-y\le  a\le  y}&\left|N_a(x)-A\pi(x)\right|^2\\
&  =\sum_{-y\le  a\le  y}N_a(x)^2-2A\pi(x)\sum_{-y\le  a\le  y}N_a(x)+A^2\pi(x)^2\sum_{-y\le  a\le  y}1\\
&  =\sum_{-y\le  a\le  y}N_a(x)^2-2A^2y\pi(x)^2+O_D\left(\frac{y\pi(x)^2}{(\log x)^{3D}}\right). 
\end{align*}
Let 
\begin{align*}
\cQ = \{(p_1,p_2)\in [1,x]^2\colon~ (\log x)^{3D+2}\le p_1,p_2&\le  x,~p_1\neq p_2,\\
&  p_1,p_2,p_1p_2 \not \in \mathcal{E}\},
\end{align*}
where $\mathcal{E}$ is the set of natural numbers $q\le  x^2$ for which
$$
\max_{\chi\in \mathcal{X}_q^{*}}\left|\sum_{1\le  a\le  y}\chi(a)\right|>\frac{y}{(\log x)^{3D+2}}.    
$$
We apply Theorem~\ref{thm:short char sum} with 
$\lambda$ as in~\eqref{eq:Cdef}. Note that this is an admissible choice, since we have $y>\exp(20\lambda \log(\lambda(\log \log(x^2)))(\log \log(x^2)))$. Now we conclude that
$$
 \#\mathcal{E}\ll x^{2\cdot 0.49}= x^{0.98}.    
$$
Hence we have
\begin{align*}
\#\mathcal{Q}=\#\{(p_1,p_2)\in [1,x]^2\colon~ (\log x)^{3D+2}\le p_1,p_2\le  x,~p_1&\neq p_2\}\\
&+O(x^{1.98}).    
\end{align*}

Again denoting by $\mathcal{G}_p$ the set of primitive roots modulo $p$, we have
$$
\sum_{-y\le  a\le  y}N_a(x)^2=\sum_{-y\le  a\le  y}\sum_{\substack{p_1,p_2\le  x\\ a\in \mathcal{G}_{p_1}\\ a\in \mathcal{G}_{p_2}}}1=\sum_{\substack{p_1,p_2\le  x\\ (p_1,p_2)\in \mathcal{Q}}}\sum_{\substack{-y\le  a\le  y\\ a\in \mathcal{G}_{p_1}\\ a\in \mathcal{G}_{p_2}}}1+O(yx^{1.98}).
$$ 
Using~\eqref{eq:primroot}, the main term in the above equation becomes
\begin{align*}
S = \sum_{\substack{p_1,p_2\le  x\\ (p_1,p_2)\in \mathcal{Q}}}\frac{\varphi(p_1-1)\varphi(p_2-1)}{(p_1-1)(p_2-1)}& \sum_{\substack{t_1\mid p_1-1\\t_2\mid p_2-1}} \frac{\mu(t_1)\mu(t_2)}{\varphi(t_1)\varphi(t_2)}\\
& \quad \sum_{\substack{\chi_1\in \mathcal{X}_{p_1}\\\ord(\chi_1)=t_1}} 
 \sum_{\substack{\chi_2\in \mathcal{X}_{p_2}\\\ord(\chi_2)=t_2}} \sum_{-y\le  a\le  y}\chi_1\chi_2(a), 
\end{align*}
where $\chi_1\chi_2(a) = \chi_1(a)\chi_2(a)$ is a now a character modulo $p_1p_2$ (since $p_1 \ne p_2$).
Let us split this as 
$$
S=\sum_{\substack{p_1,p_2\le  x\\ (p_1,p_2)\in \mathcal{Q}}}\frac{\varphi(p_1-1)\varphi(p_2-1)}{(p_1-1)(p_2-1)}\sum_{\substack{-y\le  a\le  y\\\gcd(a,p_1p_2)=1}}1+S_1+S_2+S_3,
$$
where $S_1$ corresponds to those terms with $t_1=1, t_2>1$, $S_2$ corresponds to those terms with $t_1>1, t_2=1$, and $S_3$ corresponds to those terms with $t_1,t_2>1$.  By completing the sum over $(p_1,p_2)\in \mathcal{Q}$ to all pairs $(p_1,p_2)\in [1,x]^2$ and applying~\eqref{eq:phi}, we see that
\begin{equation}\begin{split}\label{eq:Sbound}
S&=2y\left(\sum_{p\le  x}\frac{\varphi(p-1)}{p-1}\left(1-\frac{1}{p}\right)\right)^2\\
& \qquad \qquad \qquad \qquad +S_1+S_2+S_3+O_D\left(\frac{yx^2}{(\log x)^{3D+2}}\right)\\
&=2A^2y\pi(x)^2+S_1+S_2+S_3+O_D\left(\frac{y\pi(x)^2}{(\log x)^{3D}}\right). 
\end{split}
\end{equation}
The sums $S_1$ and $S_2$ are bounded symmetrically; let us bound $S_1$. Note that if $\chi_1$ is a non-principal character $\pmod{p_1}$ and $\chi_0$ is the principal character $\pmod{p_2}$, by the assumption $(p_1,p_2)\in \mathcal{Q}$ (which implies $p_1\not \in \mathcal{E}$ and $p_2\geq (\log x)^{3D+2}$), we have 
$$
\left|\sum_{-y\le  a\le  y}\chi_1\chi_0(a)\right|=\left|\sum_{-y\le  a\le  y}\chi_1(a)\right|+O\left(\frac{y}{p_2}\right)\ll_D \frac{y}{(\log x)^{3D+2}}.    
$$
Hence, we have
\begin{align*}
S_1&\ll_D \frac{y}{(\log x)^{3D+2}}\sum_{p_1,p_2\le  x}\sum_{t_1\mid p_1-1}\frac{|\mu(t_1)|}{\varphi(t_1)}\sum_{\substack{\chi_1\in \mathcal{X}_{p_1}\\\ord(\chi_1)=t_1}}1\\
&\ll  \frac{y\pi(x)}{(\log x)^{3D+2}}\sum_{p_1\le  x}2^{\omega(p_1-1)}\\
&\ll  \frac{y\pi(x)^2}{(\log x)^{3D}}   
\end{align*}
by the  Titchmarsh divisor bound~\eqref{eq:titchmarsh}. 

We are left with bounding $S_3$. Since in this sum $\chi_1\chi_2$ is a primitive character modulo $p_1p_2\le  x^2$ with $p_1p_2\not \in \mathcal{E}$, we have
\begin{align*}
 S_3&\ll_D \frac{y}{(\log x)^{3D+2}}\sum_{p_1,p_2\le  x}\sum_{\substack{t_1\mid p_1-1\\t_2\mid p_2-1}}\frac{|\mu(t_1)||\mu(t_2)|}{\varphi(t_1)\varphi(t_2)}\sum_{\substack{\chi_1\in \mathcal{X}_{p_1}\\\ord(\chi_1)=t_1\\\chi_2\in \mathcal{X}_{p_2}\\\ord(\chi_2)=t_2}}1\\
 &\ll  \frac{y}{(\log x)^{3D+2}}\left(\sum_{p\le  x}2^{\omega(p-1)}\right)^2\\
 &\ll \frac{y\pi(x)^2}{(\log x)^{3D}}
\end{align*}
by the Titchmarsh divisor bound~\eqref{eq:titchmarsh}. Now the claim follows by collecting the bounds for $S_1,S_2,S_3$ and recalling~\eqref{eq:Sbound}. 

This concludes  the proof of  Proposition~\ref{prop:L2}, which as we have mentioned,  is enough to conclude the proof of Theorem~\ref{thm:Artin Aver}.


\begin{thebibliography}{99}

\bibitem{Gar} M.~Z.~Garaev,  `Character sums in short intervals and
the multiplication table modulo a large prime', {\it Monat.
Math.\/}, {\bf 148} (2006),   127--138.

\bibitem{Goldfeld} M.~Goldfeld,  `Artin's conjecture on the average', {\it Mathematika\/} {\bf 15}(2) (1968), 223--226.

\bibitem{GuMu} R.~Gupta and M.~R.~Murty, `A remark on Artin's conjecture'. {\it Invent Math.\/} {\bf 78} (1984), 127--130.

\bibitem{HB} D.~R.~Heath-Brown, `Artin's conjecture for primitive roots', {\it Q. J. of Math.\/} {\bf 37}  (1986), 27--38.

\bibitem{Hool} C.~Hooley,  `Artin's conjecture for primitive
 roots', {\it J. Reine Angew. Math.\/} {\bf 225} (1967), 209--220.

\bibitem{IwKow} H. Iwaniec and E. Kowalski,
{\it Analytic number theory\/}, Amer.  Math.  Soc.,
Providence, RI, 2004.

\bibitem{JP}
O.~J\"arviniemi and A.~Perucca, `Unified treatment of Artin-type problems', {\it Res. Number Theory\/} {\bf 9} (2023), Art.~10.

\bibitem{JPS} 
O.~J\"arviniemi, A.~Perucca and P.~Sgobba, `Unified treatment of Artin-type problems~II', 
{\it Preprint\/}, 2022, (available from
\url{https://arxiv.org/abs/2211.15614}).

\bibitem{LN} R. Lidl and H. Niederreiter, {\it Finite fields\/},
Cambridge University Press, Cambridge, 1997.

\bibitem{Mor} P. Moree, `Artin's primitive root conjecture -- a survey', {\it Integers} {\bf 12} (2012), 1305--1416.

\bibitem{Nark} W.~Narkiewicz,  {\it Classical problems in number theory\/},
Polish Sci. Publ., Warszawa, 1986.

\bibitem{nicolas} J.~Nicolas, `Sur la distribution des nombres entiers ayant une quantit\'e fix\'ee de facteurs premiers'. {\it Acta Arith.\/} {\bf 44} (1984), 191--200.

\bibitem{PeSh}  A.~Perucca and I.~E.~Shparlinski, `Uniform bounds for the density in Artin's conjecture on primitive roots', 
{\it Bull. Lond. Math. Soc.\/}, {\bf 57} (2025), 978--991. 

\bibitem{Steph} 
P.~J.~Stephens, `An average result for Artin's conjecture', 
{\it Mathematika\/}, {\bf  16} (1969), 178--188.

\bibitem{Titchmarsh} E.~C.~Titchmarsh, `A divisor problem', {\it Rend. Circ. Mat. Palermo\/} (2), {\bf 54} (1930), 414--429; Correction: ibid. {\bf 57}
(1933), 478--479.

\bibitem{Ten} G.~Tenenbaum, {\it Introduction to analytic and
probabilistic number theory\/}, Grad. Studies Math., vol.~163, Amer. Math. Soc., 2015.

\end{thebibliography}
\end{document}